\documentclass[12pt]{amsart}
\usepackage{amsmath,amssymb,amsthm, verbatim,hyperref}
\usepackage{fullpage}
\usepackage{xcolor}
\usepackage{url}
\usepackage{mathrsfs}
\usepackage[all]{xy}
  \SelectTips{cm}{10}
  \everyxy={<2.5em,0em>:}
\usepackage{tikz}
\usepackage{picinpar} 

\DeclareMathOperator{\dist}{dist}

\DeclareMathOperator{\exc}{Exc}

\usepackage{enumerate, amsmath, amsthm, amsfonts, amssymb,  mathrsfs}

\theoremstyle{plain}
  \newtheorem{lemma}[equation]{Lemma}
  \newtheorem{proposition}[equation]{Proposition}
  \newtheorem{theorem}[equation]{Theorem}
  \newtheorem{corollary}[equation]{Corollary}
    \newtheorem{question}[equation]{Question}

    \newtheorem{conjecture}[equation]{Conjecture}

\theoremstyle{definition}
  \newtheorem{definition}[equation]{Definition}

\theoremstyle{remark}
  \newtheorem{remark}[equation]{Remark}

\renewcommand{\thesection}{\arabic{section}}
\renewcommand{\theequation}{\thesection.\arabic{equation}}

 \DeclareFontFamily{U}{manual}{}
 \DeclareFontShape{U}{manual}{m}{n}{ <->  manfnt }{}
 \newcommand{\manfntsymbol}[1]{%
    {\fontencoding{U}\fontfamily{manual}\selectfont\symbol{#1}}}

\makeatletter
   \@addtoreset{section}{part}
   \@addtoreset{equation}{section}
   \@addtoreset{footnote}{section}

    {\hspace*{\fill}$\lrcorner$\endgraf\endgroup\end{trivlist}}
 
\makeatother

  \DeclareFontFamily{OT1}{pzc}{}
  \DeclareFontShape{OT1}{pzc}{m}{it}{<-> s * [1.100] pzcmi7t}{}
  \DeclareMathAlphabet{\mathpzc}{OT1}{pzc}{m}{it}

\newif\ifhascomments \hascommentstrue
\ifhascomments
  \newcommand{\david}[1]{{\color{red}[[\ensuremath{\bigstar\bigstar\bigstar} #1]]}}
  \newcommand{\matt}[1]{{\color{red}[[\ensuremath{\spadesuit\spadesuit\spadesuit} #1]]}}
    \newcommand{\brian}[1]{{\color{blue}[[\ensuremath{\clubsuit\clubsuit\clubsuit} #1]]}}
\else
  \newcommand{\david}[1]{}
  \newcommand{\matt}[1]{}
    \newcommand{\brian}[1]{}
\fi

\newcommand{\A}{\mathcal A}

\newcommand{\C}{\mathcal C}

\renewcommand{\emptyset}{\varnothing}

\DeclareMathOperator{\Endo}{\ensuremath{\mathcal{E}\kern-.125em\mathpzc{nd}}}

\newcommand{\F}{\mathcal F}

\newcommand{\G}{\mathcal G}

\DeclareMathOperator{\Hom}{\ensuremath{\mathcal{H}\kern-.125em\mathpzc{om}}}
\newcommand{\id}{\mathrm{id}}

\newcommand{\M}{\mathcal M}
\newcommand{\N}{\mathcal N}

\renewcommand{\O}{\mathcal O}
\DeclareMathOperator{\Pic}{Pic}

\newcommand{\PP}{\mathbb{P}}

\newcommand{\QQ}{\mathbb Q}
\newcommand{\RR}{\mathbb R}

\renewcommand{\setminus}{\smallsetminus}

\newcommand{\Z}{\mathcal{Z}}

 \def\ari[#1]{\ar@{^(->}[#1]}
 \def\are[#1]{\ar[#1]^{\txt{\'et}}}
 \def\areh[#1]{\ar[#1]|{\txt{$H$-eq}}^{\txt{\'et}}}
 \def\ars[#1]{\ar@{->>}[#1]}
 \newcommand{\dplus}{\ar@{}[d]|{\mbox{$\oplus$}}}
 \newcommand{\dtimes}{\ar@{}[d]|{\mbox{$\times$}}}

\usepackage{amsthm}

\theoremstyle{plain}

\newtheoremstyle{named}{}{}{\itshape}{}{\bfseries}{.}{.5em}{\thmnote{#3 }#1}
\theoremstyle{named}

\usepackage{amsmath}

\begin{document}
\newtheorem{thm}{Theorem}[section]
\newtheorem{lem}[thm]{Lemma}
\newtheorem{dfn}[thm]{Definition}
\newtheorem{cor}[thm]{Corollary}
\newtheorem{conj}[thm]{Conjecture}
\newtheorem{clm}[thm]{Claim}
\theoremstyle{remark}
\newtheorem{exm}[thm]{Example}
\newtheorem{rem}[thm]{Remark}
\newtheorem{que}[thm]{Question}
\def\N{{\mathbb N}}
\def\G{{\mathbb G}}
\def\F{{\mathbb F}}
\def\Q{{\mathbb Q}}
\def\R{{\mathbb R}}
\def\C{{\mathbb C}}
\def\P{{\mathbb P}}
\def\A{{\mathbb A}}
\def\Z{{\mathbb Z}}
\def\v{{\mathbf v}}
\def\w{{\mathbf w}}
\def\x{{\mathbf x}}
\def\O{{\mathcal O}}
\def\M{{\mathcal M}}
\def\kbar{{\bar{k}}}
\def\tr{\mbox{Tr}}
\def\id{\mbox{id}}

\title{Approximating rational points on surfaces}

\author{Brian Lehmann}
\address{Department of Mathematics \\
Boston College  \\
Chestnut Hill, MA \, \, 02467}
\email{lehmannb@bc.edu}

\author{David McKinnon}
\address{University of Waterloo \\
Department of Pure Mathematics \\
Waterloo, Ontario \\
Canada  N2L 3G1}
\email{dmckinnon@uwaterloo.ca}

\author{Matthew Satriano}
\address{University of Waterloo \\
Department of Pure Mathematics \\
Waterloo, Ontario \\
Canada  N2L 3G1}
\email{msatrian@uwaterloo.ca}

\thanks{The first author was supported by Simons Foundation grant Award Number 851129.  The second and third authors were partially supported by Discovery Grants from the Natural Sciences and Engineering Research Council. The third author was additionally supported by a Mathematics Faculty Research Chair.}

\begin{abstract}
Let $X$ be a smooth projective algebraic variety over a number field $k$ and $P\in X(k)$. In 2007, the second author conjectured that, in a precise sense, if rational points on $X$ are dense enough, then the best rational approximations to $P$ must lie on a curve. We present a strategy for deducing a slightly weaker conjecture from Vojta's conjecture, and execute the strategy for the full conjecture for split surfaces.
\end{abstract}

\maketitle
\tableofcontents

\section{Introduction}

A foundational result of Dirichlet states that if $x$ is an irrational number then there are infinitely many rational numbers $\frac{a}{b}$ in lowest terms satisfying the equation $|x-\frac{a}{b}|<\frac{1}{b^2}$. Said another way, for any $x\in\RR$ we have an {\em approximation exponent} $\tau_{x}\in (0,\infty]$ which is the unique number such that
$$\left|{x-\frac{a}{b}}\right| \leq \frac{1}{b^{\tau_{x}+\delta}}$$
has only finitely many solutions $\frac{a}{b}\in \QQ$ whenever $\delta>0$, and has infinitely many solutions whenever $\delta<0$.  
Dirichlet's aforementioned Approximation Theorem then states $\tau_x\geq2$ for irrational $x$. On the other hand, for $x\in \RR$ algebraic of degree $d$ over $\QQ$ Liouville \cite{L} proved the upper bound $\tau_x\leq d$. This bound was subsequently improved by Thue \cite{thue-approx}, Siegel \cite{siegel-approx}, and independently by Dyson \cite{dyson-approx} and Gelfand, culminating in Roth's 1955 theorem \cite{roth-rat-approx} that $\tau_x\leq2$ for all algebraic $x\in\RR$. Therefore, Dirichlet's Theorem and Roth's Theorem combine to show that $\tau_x=2$ for all irrational $x$.

McKinnon and Roth \cite{MR} generalized the definition of $\tau_x$ to polarized projective varieties $(X,A)$ over a number field $k$ by replacing the function $|x-\frac{a}{b}|$ by a distance function $\dist_v(x,\cdot)$ depending on a place $v$ of $k$, and replacing $b^{-\tau_x}$ by a quantity coming from the height function $H_A(\cdot)$ associated to $A$. In order to make this set-up better behaved with respect to changes in $A$, they moved the exponent $\tau_x$ from the height to the distance and relabeled it as $\alpha$. Given any sequence of rational points $\{x_i\}$ approximating an algebraic point $P$, one then obtains an associated \emph{approximation constant} $\alpha(\{x_i\},P,A)$; see Section \ref{sec:prelim} for the precise definition.

\begin{definition}
Let $(X,A)$ be a projective variety over a number field $k$ equipped with a $\mathbb{Q}$-Cartier divisor $A$.  For any algebraic  point $P$, we let
\begin{itemize}
\item the \emph{approximation constant} $\alpha(P,A)$ denote the infimum of $\alpha(\{x_i\},P,A)$ over all sequences of $k$-rational points $\{ x_{i} \}$ approximating $P$,
\item the \emph{essential approximation constant} $\alpha_{ess}(P,A)$ denote the infimum of all $\alpha(\{x_i\},P,A)$ over all sequences of $k$-rational points $\{ x_{i} \}$ approximating $P$ such that $\{x_i\}$ is Zariski dense.
\end{itemize}
Lastly, if $P$ lies on a subvariety $i\colon Z\hookrightarrow X$, then considering $P$ as a point on $Z$, we define $\alpha(P,A|_Z):=\alpha(P,i^*A)$.
\end{definition}

The focus of our paper is a conjecture introduced by the second author:

\begin{conjecture}[{\cite[Conjecture 2.7]{M2}}]
\label{conj:ratcurve}
	Let $X$ be a smooth projective variety defined over a number field $k$, and $P\in X$ a $k$-rational point.  Let $A$ be an ample $\mathbb{Q}$-Cartier divisor on $X$.  If $\alpha(P,A)<\infty$, then there is a curve of best $A$-approximation to $P$.  In other words, there is a curve $C$ such that
	\[\alpha(P,A|_C)=\alpha(P,A)\]
\end{conjecture}

In \cite{M2}, this conjecture was proved for projective space, for split smooth rational surfaces of Picard rank at most four, and for a number of other split rational surfaces. In \cite{MR}, the second author and Roth proved the conjecture for split smooth cubic surfaces and gave the first cases of the conjecture for non-split surfaces. Casta\~neda in \cite{Ca} proved many more instances of the conjecture for split rational surfaces, including all split smooth rational surfaces of Picard rank at most five. Huang proved Conjecture \ref{conj:ratcurve} for some split toric varieties in \cite{Hu}, providing the first nontrivial examples in higher dimensions. In \cite{MS}, the second and third authors proved that Vojta's Conjecture implies Conjecture \ref{conj:ratcurve} for all $\Q$-factorial split toric varieties.

In this paper, we prove Conjecture \ref{conj:ratcurve} for all split surfaces contingent on Vojta's conjecture. Furthermore, using the Minimal Model Program (MMP), we provide a general framework which we expect will be useful in handling higher dimensional cases of the conjecture.

We introduce the following definition.

\begin{definition}\label{dfn:essentiallybounded}
	Let $X$ be a smooth algebraic variety defined over a number field $k$, and let $P\in X$ be any $k$-rational point.  Let $K_X$ be the canonical divisor class on $X$.  Then $P$ is {\em essentially bounded on $X$} if for every ample divisor $A$ and every rational $\epsilon>0$, we have
	\[\alpha_{ess}(P,-K_X+\epsilon A)\geq\dim X \]
\end{definition}

Our first result is that Vojta's Conjecture implies every point is essentially bounded. The following proposition is a variant of \cite{MS-erratum}.

\begin{proposition}\label{prop:Vojta->ess-boudned}
Let $X$ be a smooth, projective variety defined over a number field $k$, and let $P\in X(k)$.  Then Vojta's Conjecture for $X$ implies that $P$ is essentially bounded.
\end{proposition}

The main theorem of our paper is that Conjecture~\ref{conj:ratcurve} is true for essentially bounded points on split surfaces. We say a surface $X$ over $k$ is \emph{split} if the Galois action on $\Pic(X)$ is trivial. We note that this is a more general definition than other ones in the literature; for example, it includes Brauer-Severi varieties.

\begin{theorem}\label{thm:generalsurface}
Let $X$ be a split, smooth, projective surface defined over a number field $k$. Let $A$ be an ample $\Q$-Cartier divisor on $X$ and let $P\in X(k)$. If $P$ is essentially bounded on $X$, then Conjecture~\ref{conj:ratcurve} is true for $P$ on $X$ with respect to $A$.

In particular, for an arbitrary smooth projective surface $X$, there is a finite extension $\ell/k$ such that Conjecture~\ref{conj:ratcurve} is true for all $A$ and all essentially bounded points of $X\times_k\ell$.
\end{theorem}

Putting Proposition \ref{prop:Vojta->ess-boudned} and Theorem \ref{thm:generalsurface} together, we have

\begin{corollary}\label{thm:generalsurface-Vojta-cor}
Let $X$ be a split, smooth, projective surface defined over a number field $k$. If Vojta's Conjecture is true for $X$, then Conjecture~\ref{conj:ratcurve} is true for all $P$ and $A$ on $X$.
\end{corollary}

%

Let us describe the strategy of proof for Theorem \ref{thm:generalsurface}. By Theorem~4.2 of \cite{M2}, we are able to reduce to the case of negative Kodaira dimension. In this case, we find a curve $C$ such that $\alpha(P,A|_C) \leq \alpha_{ess}(P,A)$.  This tells us that no Zariski dense sequence can have an $\alpha$-value less than that computed on $C$, and so the $\alpha$-value at $P$ must be computed by a (possibly different) curve $C'$.  That is, we have $\alpha(P,A|_{C'})=\alpha(P,A)$, proving Conjecture \ref{conj:ratcurve}.

To establish the bound $\alpha(P,A|_C) \leq \alpha_{ess}(P,A)$, we use tools in birational geometry.  We begin by rescaling the ample divisor $A$ so that $K_X+A$ is on the boundary of the pseudo-effective cone.  We then seek curves with small intersection against both $-K_{X}$ and $K_{X}+A$, thereby  achieving our goal of finding a curve with small intersection against $A$ (and thus with a small value of $\alpha(P,A|_{C})$).  Since Mori's Bend-and-Break theorem allows us to find rational curves with small intersection against $-K_{X}$, curves contracted by the $(K_{X} + A)$-MMP are good candidates for our desired curves.

By running the MMP for the pair $(X,A)$, we obtain a birational morphism $\varphi\colon X\to X'$ and a canonical model $\pi: X' \to Z$. If $P$ is in the $\varphi$-exceptional locus, we show that an exceptional curve satisfies the desired conclusion of the conjecture. Otherwise, we produce a rational curve $C'$ through $\varphi(P)$ with small anti-canonical degree and show that its strict transform $C$ on $X$ satisfies $\alpha(P,A|_C)\leq 2$. Since essential boundedness implies that $2\leq\alpha_{ess}(P,A)$, this curve $C$ achieves our goal. 

Much of our set-up applies in higher dimensions as well:~at least over $\overline{k}$ one can use the MMP to find rational curves of low degree through rational points $P$ on $X$.  
Just as in the case of surfaces, the existence of a curve $C$ with $\alpha(P,A|_C) \leq \alpha_{ess}(P,A)$ would imply that the $\alpha$-value is computed on a proper subvariety.  Unlike the case of surfaces, we cannot directly conclude that $\alpha$ is computed by a curve $C$.  Thus in higher dimensions the strategy is best suited for establishing the following weaker conjecture.

\begin{conjecture}\label{conj:degenerate}
	Let $X$ be a smooth, projective variety defined over a number field $k$, and $P\in X$ a $k$-rational point.  Let $A$ be an ample line bundle on $X$.  If $\alpha(P,A)<\infty$, then there is a sequence of best $A$-approximation to $P$ that is not Zariski dense.  In other words, there is a proper subvariety $V$ such that
	\[\alpha(P,A|_V)\leq\alpha_{ess}(P,A)\]
\end{conjecture}
\begin{remark}
Conjectures \ref{conj:ratcurve} and \ref{conj:degenerate} are equivalent in the case of surfaces.
\end{remark}

The structure of the paper is as follows.  Section~\ref{sec:prelim} defines the approximation constants and gives some basic properties.  Section~\ref{sec:strategy} describes the strategy of the proof in all dimensions using the Minimal Model Program.  Section~\ref{sec:surfaces} gives the proof for surfaces.

\section{Preliminaries}\label{sec:prelim}

Throughout the rest of the paper, $k$ is a fixed number field, and $v$ a fixed place of $k$.  We further fix an extension of $v$ to $\overline{k}$, which we abusively also call $v$.  We use this place only to fix a distance function $\dist(\cdot,\cdot)$ on the rational points of $X$, and we generally suppress any further mention of it.  Also, height functions defined below are only defined up to multiplication by a positive bounded function.  These heights are only ever used to compute the value of $\alpha$, where this ambiguity is unimportant, so we will suppress precise definitions of the height functions.

We first define the basic measure of approximation on $X$.  

\begin{definition}\label{dfn:seqappconst}
	Let $X$ be a projective variety, $P\in X(\overline{k})$, $D$ a $\mathbb{Q}$-Cartier divisor on $X$.  For any sequence $\{x_i\}\subseteq X(k)$ of distinct
	points with $\dist(P,x_i)\rightarrow 0$, we set
	$$A(\{x_i\}, D) = \left\{{
		\gamma\in\R \mid
		\dist(P,x_i)^{\gamma} H_{D}(x_i)\,\,\mbox{is bounded from above}
	}\right\}.
	$$
	If $\{x_i\}$ does not converge to $x$ then we set $A(\{x_i\},D)=\emptyset$.
\end{definition}

It follows easily from the definition that if $A(\{x_i\}, L)$ is nonempty then it is an
interval unbounded to the right, i.e., if $\gamma\in A(\{x_i\},L)$ then $\gamma+\delta\in A(\{x_i\},L)$ for any $\delta>0$.

\begin{definition}
	For any sequence $\{x_i\}$ we set $\alpha(\{x_i\},P,L)$ to be the infimum of $A(\{x_i\},L)$
	(in particular if $A(\{x_i\},L)=\emptyset$ then $\alpha(\{x_i\},P,L)=\infty$).
	We call $\alpha(\{x_i\},P,L)$ the approximation constant of $\{x_i\}$ with respect to $L$.
\end{definition}

%

\begin{definition}\label{dfn:approx}
	The approximation constant $\alpha(P,L)$ of $P$ with respect to
	$L$ is defined to be the infimum of all approximation constants of
	sequences of points in $X(k)$ converging $v$-adically to $P$.
	If no such sequence exists, we set $\alpha(P,L)=\infty$.
\end{definition}

\begin{remark}\label{rem:scaling}
	Note that $\alpha$ satisfies
\[\alpha(P,aL)=a\alpha(P,L)\]
for any positive integer $a$ and so the definition extends to $\Q$-Cartier divisors.
\end{remark}

For further details on $\alpha$, we refer the reader to \cite{MR}.  We will record two useful observations here.

\begin{proposition}[{\cite[Proposition 2.14(b)]{MR}}]\label{prop:concavity}
	Let $\{x_i\}$ be any sequence of $k$-rational points converging to $P$.  Then for any positive rational numbers $a$ and $b$, and any $\Q$-Cartier divisors $L_1$ and $L_2$ (with the exception of the case that $\{\alpha(\{x_i\},P,L_1),\alpha(\{x_i\},P,L_2)\} = \{-\infty,\infty\}$) we have
	\[\alpha(\{x_i\},P,a L_1+ b L_2) \geq a\alpha(\{x_i\},P,L_1)+b\alpha(\{x_i\},P,L_2)\]
\end{proposition}

\begin{lemma}[{\cite[Lemma 4.7]{MS}}]
	\label{l:alpha-unchanged-under-pullbacks}
	Let $\pi\colon X\to Y$ be a surjective birational morphism of projective $\QQ$-factorial varieties over a number field $k$. Let $P\in X(k)$ be a point which is not in the exceptional locus $\exc(\pi)$ and let $D'$ be a $\QQ$-Cartier $\Q$-divisor on $Y$.  Suppose either that $\{x_i\}$ is a Zariski dense sequence on $X$ converging to $P$ and let $x'_i:=\pi(x_i)$, or suppose $\{x'_i\}$ is a Zariski dense sequence on $X'$ converging to $\pi(P)$ and let $x_i:=\pi^{-1}(x'_i)$ whenever $x'_i\notin\pi(\exc(\pi))$. Then $\alpha_{P,\{x_i\}}(\pi^*D')=\alpha_{\pi(P),\{x'_i\}}(D')$.
\end{lemma}

\begin{definition}\label{dfn:essentialapproximationconstant}
	The \emph{essential approximation constant} of a variety $V$ over a number field $k$ with respect to a point $P$ and a $\Q$-Cartier divisor $L$ on $V$ is
	\[\alpha_{ess}(P,L)=\inf_{\{x_i\}\,\mbox{\tiny{Zariski dense}}}\alpha(\{x_i\},P,L)\]
	where $\{x_i\}$ ranges over all Zariski dense sequences that converge to $P$.
	In other words, it is the infimum of all the approximation constants of Zariski dense sequences.
\end{definition}

Given the above definition, the following corollary of Proposition~\ref{prop:concavity} is immediate.

\begin{corollary}\label{cor:concavity}
	For any $\Q$-Cartier divisors $L_1$ and $L_2$ and any rational numbers $a$ and $b$, we have
	\[\alpha(P,a L_1+ b L_2) \geq a\alpha(P,L_1)+b\alpha(P,L_2)\]
	unless $\{\alpha(P,L_1),\alpha(P,L_2)\}=\{-\infty,\infty\}$, and
	\[\alpha_{ess}(P,a L_1+ b L_2) \geq a\alpha_{ess}(P,L_1)+b\alpha_{ess}(P,L_2)\]
	unless $\{\alpha_{ess}(P,L_1),\alpha_{ess}(P,L_2)\}=\{-\infty,\infty\}$.
\end{corollary}


As stated in Proposition \ref{prop:Vojta->ess-boudned}, Vojta's conjectures imply that every $k$-rational point $P$ is essentially bounded.  We will prove this result after first recalling for the reader's convenience the statement of Conjecture~\ref{conj:vstrong} from \cite{Vo}.  We refer the reader to \cite{Vo} for the definition of the various quantities used in the statement of the conjecture.

\begin{conjecture}[Vojta's Main Conjecture]\label{conj:vstrong}
	Let $X$ be a smooth, projective variety defined over a number field $k$, with canonical divisor $K_X$.  Let $S$ be a finite set of places of $k$.  Let $A$ be a big divisor on $X$, and let $D$ be a normal crossings divisor on $X$.  Choose height functions $h_{K_X}$ and $h_A$ for $K_X$ and $A$, respectively, and define a proximity function $m_S(D,P) = \sum_{v\in S} h_{D,v}(P)$ for $D$ with respect to $S$, where $h_{D,v}$ is a local height function for $D$ at $v$.  Choose any $\epsilon>0$.  Then there exists a nonempty Zariski open set $U=U(\epsilon)\subseteq X$ such that for every $k$-rational point $Q\in U(k)$, we have the following inequality:
	\begin{equation}\label{vojta-ineq}
		m_S(D,Q) + h_{K_X}(Q) \leq \epsilon h_A(Q) + B
	\end{equation}
where $B$ is a constant that does not depend on $Q$.
\end{conjecture}

We are now ready to prove Proposition \ref{prop:Vojta->ess-boudned}.



\begin{proof}[{Proof of Proposition \ref{prop:Vojta->ess-boudned}}]
Choose any rational $\epsilon>0$ and any ample divisor $A$ on $X$.  Let $D$ be a normal crossings divisor defined over $k$ with multiplicity $\dim X$ at $P$.  (For example, $D$ could be the union of $\dim X$ transversely intersecting smooth divisors containing $P$ in the class of $mA$ for some large $m$.)  Let $U(\epsilon)$ be chosen as in Conjecture \ref{conj:vstrong} and suppose $Q \in U(\epsilon)$ is a $k$-rational point.  Due to our choice of $D$, for some additive constants $M$ and $M'$ not depending on $Q$ we have:
\[m_S(D,Q)=-\log\dist(D,Q)+M\geq -\log\dist(P,Q)^{\dim X}+M'\]

Plugging this into inequality \ref{vojta-ineq}, exponentiating, and rearranging, we find that for some constant $N$ not depending on $Q$:
\[\dist(P,Q)^{\dim X}H_{-K_X+\epsilon A}(Q)\geq N.\]

This means that for every sequence $\{x_i\}$ of $k$-rational points in $U(\epsilon)$ with $\dist(P,x_i)\to 0$, and for every $\gamma<\dim X$, the quantity
\[\dist(P,x_i)^\gamma H_{-K_X+\epsilon A}(x_i)\]
goes to infinity as $i\to\infty$.  Therefore for any sequence in $U(\epsilon)$, we have
\[\alpha(\{x_i\},P,-K_X+\epsilon A)\geq\dim X\]

Since any Zariski dense sequence must intersect $U(\epsilon)$ in a cofinal subsequence, this immediately implies that 
\[\alpha_{ess}(P,-K_X+\epsilon A)=\inf_{\{x_i\}\,\mbox{\tiny{Zariski dense}}}\alpha(\{x_i\},P,L)\geq\dim X\]
as desired.
\end{proof}

\section{General Results}\label{sec:strategy}

In this section, we describe the strategy for proving Conjecture \ref{conj:degenerate} in full generality.  Our first step is to notice that we can rescale our ample divisor $A$.

\begin{proposition}\label{prop:scaling}
	Let $A$ be an ample $\Q$-Cartier divisor and $y$ a positive rational number.  If Conjecture \ref{conj:ratcurve} is true for $A$, then it is true for $yA$.
\end{proposition}

\begin{proof}
	This follows immediately from the fact that $\alpha(\{x_i\},P,yA)=y\alpha(\{x_i\},P,A)$ for any sequence $\{x_i\}$.
\end{proof}

\begin{proposition}\label{prop:stepone}
	Let $X$ be a smooth, projective variety defined over a number field $k$, $P$ a $k$-rational point of $X$, $A$ an ample $\Q$-Cartier divisor on $X$ such that $K_X+A$ is in the closure of the effective cone.  Then if $P$ is essentially bounded, we have
	\[\alpha_{ess}(P,A)\geq\dim X\]
\end{proposition}

\begin{proof}
	We first note that by Theorem 1.1 of \cite{BCHM}, some multiple of $K_X+A$ is effective.  Applying for example Theorem B.3.2.(e) of \cite{HS}, we conclude there is a constant $N>0$ such that for all $k$-rational $Q$ not in the support of $K_X+A$, we have
	\[H_{K_X+A}(Q)\geq N\]
	In particular, for every $\epsilon>0$, we get
	\[H_{(1+\epsilon)A}(Q)\geq NH_{-K_X+\epsilon A}(Q)\]
	and therefore
	\[\alpha_{ess}(P,(1+\epsilon)A)\geq\alpha_{ess}(P,-K_X+\epsilon A)\geq\dim X\]
	by essential boundedness.  Since $\alpha_{ess}(P,(1+\epsilon)A)=(1+\epsilon)\alpha_{ess}(P,A)$, as $\epsilon\to 0$ we deduce the proposition.
\end{proof}


It remains to find a curve that contains a sequence of approximation whose $\alpha$-value is no larger than $\dim X$.  To do this, we need to be able to calculate $\alpha$ on a curve. For this we make use of the following result.

\begin{theorem}[{\cite[Theorem 2.16]{MR}}]
	\label{thm:curve}
	Let $C$ be any rational curve defined over $k$ and $\varphi\colon\PP^1\rightarrow C$ the normalization map.
	Then for any nef line bundle $L$ on $C$, and any $x\in C(\overline{k})$ we have the equality:
	\[\alpha_{x,C}(L)=\min_{q\in \varphi^{-1}(x)} d/r_{q} m_{q}\]
	where $d=\deg(L)$, $m_{q}$ is the multiplicity of the branch of $C$ through $x$ corresponding to $q$, and
	\[r_{q}=
	\begin{cases}
		0 & \text{if $\kappa(q)\not\subseteq k_v$} \\
		1 &\text{if $\kappa(q)=k$} \\
		2 &\text{otherwise.}
	\end{cases}
	\]
\end{theorem}

For the purposes of this paper, the relevant consequence of this theorem is the following corollary.

\begin{corollary}\label{cor:curve}
	Let $X$ be a projective variety defined over a number field $k$, and let $C$ be a rational curve on $X$, also defined over $k$.  Let $A$ be a $\Q$-Cartier divisor on $X$ such that $A \cdot C \geq 0$.  For any $k$-rational point $P\in C(k)$ that is a $v$-adic limit of other $k$-rational points on $C$, we have
	\[\alpha(P,A|_C)=MA\cdot C\]
	where $M$ is a rational number independent of $A$ with $0\leq M\leq 1$.
\end{corollary}

\begin{proof}
	We apply Theorem~\ref{thm:curve} with $d=A.C$.  Thus $\alpha(P,A|_C)=\min_q\{(A.C)/r_qm_q\}\leq A.C$ unless $r_q=0$ for all $q$.  Since $r_q=0$ for all $q$ if and only if $P$ is $v$-adically isolated on $C$, the corollary follows with $M=\min_q\{1/r_qm_q\}$. 
\end{proof}

\begin{remark}
	Note that if the curve $C$ is smooth at $P$, then $M=1$.
\end{remark}

The following is a variant of \cite[Lemma 4.3]{MS}. 

\begin{lemma}
	\label{l:reduce-to-a=0}
	Let $X$ be a projective, $\QQ$-factorial algebraic variety over a number field $k$ which is essentially bounded at $P\in X(k)$. Let $A$ be an ample $\Q$-divisor on $X$ such that $K_X+A$ is nef. Suppose $C$ is an irreducible curve through $P$ with $-K_X\cdot C\leq\dim X$ and
	\[\alpha(P,(K_X+A)|_C)\leq \alpha_{ess}(P,K_X+A) < \infty.\] 
	Then we have 
	\[\alpha(P,A|_C)\leq \alpha_{ess}(P,A)\]
\end{lemma}

\begin{proof}
	Since $\alpha(P,(K_X+A)|_C) < \infty$, the curve $C$ must be rational and $P$ must be a $v$-adic limit of rational points on $C$, so Corollary~\ref{cor:curve} gives us that $\alpha(P,F|_C)=MF \cdot C$ for every nef $\Q$-divisor $F$, where $M$ is independent of $F$ and satisfies $0 \leq M \leq 1$.   In particular, 
	\[
	\alpha(P,A|_C)=MA\cdot C=M(K_X+A)\cdot C-MK_X\cdot C\leq \alpha(P,(K_X+A)|_C)+\dim X.
	\]
	Using the defining property of $C$ and the fact that $X$ is essentially bounded at $P$, we see that for an arbitrary $\epsilon>0$ and ample $B$, we have
	\[
	\alpha(P,A|_C)\leq \alpha_{ess}(P,K_X+A)+\alpha_{ess}(P,-K_X+\epsilon B).
	\]
	Lastly, concavity of $\alpha_{ess}$, shown in Corollary~\ref{cor:concavity}, yields 
	\[
	\alpha(P,A|_C)\leq \alpha_{ess}(P,A+\epsilon B).
	\]
	Since $\alpha_{ess}(P,\cdot)$ is continuous at $A$ (it is concave on the ample cone), we are done.
\end{proof}

\subsection{Running the MMP}
Using the convexity of $\alpha$, one can bound the essential approximation constant for an ample divisor $A$ using the essential approximation constants for $K_{X} + A$ and for $-K_{X}$.  The latter is controlled by the essential boundedness assumption.  We can use the minimal model program to reduce the computation of the former to an easier situation.

When $X$ is a smooth, projective variety and $A$ is an ample $\mathbb{Q}$-Cartier divisor on $X$ then \cite{BCHM} guarantees that we can run the $(K_{X}+A)$-MMP.  The outcome is described by the following definition:

\begin{definition}
Let $X$ be a normal $\mathbb{Q}$-factorial projective variety and let $D$ be a $\mathbb{Q}$-Cartier divisor on $X$.  A $D$-negative birational contraction is a rational map $\varphi: X \dashrightarrow X'$ satisfying the following properties:
\begin{enumerate}
\item The inverse map $\varphi^{-1}$ does not extract any divisors.
\item The divisor $D' = \varphi_{*}D$ is $\mathbb{Q}$-Cartier.
\item Let $X\stackrel{\Phi}{\leftarrow} W\stackrel{\Phi'}{\rightarrow} X'$ denote any smooth birational model that resolves the map $\varphi$.  Then we can write $\Phi^{*}(D) = \Phi'^{*}(D') + E$ where $E$ is an effective $\Phi'$-exceptional divisor whose support contains every divisor contracted by $\Phi'$ but not by $\Phi$.
\end{enumerate}
Lemma 3.38 of \cite{KM98} and Lemma 3.6.3 of \cite{BCHM} show that any sequence of birational steps in the $(K_{X}+A)$-MMP will define a $(K_{X}+A)$-negative birational contraction $\varphi: X \dashrightarrow X'$.  
\end{definition}

Next, we must relate rational approximation on $X$ with rational approximation on $X'$. 
We consider the behavior of approximation constants along curves.

\begin{proposition} \label{prop:stricttransformcalculation}
Let $X$ be a smooth, geometrically integral, projective variety defined over a number field $k$.  Let $\varphi: X \dashrightarrow X'$ be a birational contraction.  Suppose that $P \in X(k)$ lies in the locus where $\varphi$ is an isomorphism and let $P'$ denote its image in $X'$.

Suppose $C'$ is a $k$-rational curve in $X'$ containing $P'$.  Then the strict transform $C$ of $C'$ is a $k$-rational curve in $X$ containing $P$ which satisfies
\begin{equation*}
\alpha(P,A|_{C}) \leq \alpha(P',(\varphi_{*}A)|_{C'}).
\end{equation*}
where $A$ is any ample $\mathbb{Q}$-divisor on $X$.
\end{proposition}

\begin{proof}
	First, note that $A\cdot C\leq (\varphi_*A)\cdot C'$.  This is because we can choose a divisor linearly equivalent to $A$ whose intersection with $C$ is entirely supported on the open set on which $\varphi$ is an isomorphism.   

We now appeal to Theorem \ref{thm:curve}.  For a rational point $P$ on $C$ the approximation constant $\alpha(P,A|_{C})$ is determined by the intersection number $A \cdot C$ and the multiplicities of the branches of $C$ through $P$.  Since $\varphi$ is an isomorphism on a neighborhood of $P$, the branch multiplicities of $C'$ do not change upon taking a strict transform.  Then the description of $\alpha(P,A|_{C})$ in Theorem \ref{thm:curve} yields the desired statement.
\end{proof}

This discussion yields the following result.

\begin{theorem} \label{theo:mmpcomputedonsubvar}
Let $X$ be a smooth, geometrically integral, projective variety defined over a number field $k$.  Suppose that $P \in X(k)$ is a essentially bounded point and that $A$ is an ample $\mathbb{Q}$-Cartier divisor on $X$ such that $(K_{X} + A)$ is pseudo-effective.   Let $\varphi: X \dashrightarrow X'$ be a $(K_{X}+A)$-negative birational contraction such that $P$ is contained in the locus where $\varphi$ is an isomorphism.

Suppose that there is a rational curve $C'$ through $\varphi(P)$ such that 
\begin{equation*}
\varphi_{*}A \cdot C' \leq \dim(X)
\end{equation*}
and such that $\varphi(P)$ is a $v$-adic limit of $k$-rational points.  Then the strict transform $C$ of $C'$ satisfies
\begin{equation*}
\alpha(P,A|_{C}) \leq \alpha_{ess}(P,A).
\end{equation*}
\end{theorem}

\begin{proof}
By Proposition~\ref{prop:stricttransformcalculation} we see that \[\alpha(P,A|_C)\leq\alpha(P,(\varphi_*A)|_{C'}).\]  
Since $\varphi_{*}A \cdot C' \leq \dim(X)$ by assumption, it follows by Theorem~\ref{thm:curve} that 
\[\alpha(P,(\varphi_*A)|_{C'})\leq\dim X\]
But by Proposition~\ref{prop:stepone}, we have
\[\alpha_{ess}(P,A)\geq\dim X.\]
\end{proof}

\subsection{Finding rational curves}

Let $X$ be a smooth, geometrically integral, projective variety defined over a number field $k$ and let $A$ be an ample $\mathbb{Q}$-Cartier divisor on $X$ such that $K_{X}+A$ is pseudo-effective.   Given an essentially bounded point $P \in X(k)$, we would like to find a $k$-rational curve $C$ such that $\alpha(P,A|_{C}) \leq \alpha_{ess}(P,A)$.  According to Theorem~\ref{theo:mmpcomputedonsubvar}, we should find a run $\varphi: X \dashrightarrow X'$ of the $(K_{X}+A)$-MMP that is an isomorphism on a neighborhood of $P$ and a $k$-rational curve $C'$ satisfying $\varphi_{*}A \cdot C' \leq \dim(X)$.

Note that the inequality $\varphi_{*}A \cdot C' \leq \dim(X)$ is implied by the following two inequalities:
\begin{equation*}
(K_{X'} + \varphi_{*}A) \cdot C' \leq 0 \qquad \qquad -K_{X'} \cdot C' \leq \dim(X).
\end{equation*}
The first inequality will automatically be satisfied by curves $C'$ contracted by the next step of the $(K_{X}+A)$-MMP.  The second inequality will be satisfied when these contracted curves have small anticanonical degree.

Altogether we obtain the following strategy in higher dimensions: suppose $P \in X(k)$.  Run the $(K_{X}+A)$-MMP until we reach a step $\varphi: X \dashrightarrow X'$ such that $P$ lies in contracted locus for the next step.  If we can find a contracted rational curve $C'$ of small anticanonical degree that contains $P$ (and such that $P$ is a $v$-adic limit of $k$-rational points), its strict transform $C$ will satisfy $\alpha(P,A|_C)\leq\alpha_{ess}(P,A)$ on $X$, verifying Conjecture~\ref{conj:degenerate} for $P$ on $X$.

When $X' \cong \mathbb{P}^{n}$ then there is no curve with anticanonical degree $\leq \dim(X)$ and so we need a different argument, given by the following lemma.  Note that any birational contraction from $X$ to $\P^n$ must be a divisorial contraction.

\begin{lemma}
	\label{l:div-contraction-case-pn}  
	Suppose that $X$ is a smooth projective variety, $\psi: X \to \PP^n$ is a birational morphism, and $P \subset X(k)$ is an essentially bounded point not contained in the contracted locus.  Then there is a smooth irreducible curve $C$ through $P$ such that $-K_X\cdot C\leq\dim X$ and therefore by Theorem~\ref{theo:mmpcomputedonsubvar}, $\alpha(P,A|_C)\leq \alpha_{ess}(P,A)$.
\end{lemma}

\begin{proof} 
	Let $Z\subset \PP^n$ be the sublocus where $\psi^{-1}$ is not an isomorphism.  Let $\ell$ be a line in $\PP^n$ that contains both $\psi(P)$ and at least one point of $Z$.  Since $X$ and $\PP^{n}$ are smooth, the exceptional locus for $\psi$ is divisorial.  In particular, letting $C$ be the strict transform of $\ell$ on $X$ we see that $C\cdot E\geq 1$ for some $\psi$-contracted divisor $E$.  
	
	Since $K_X-\psi^*K_{\PP^n}-E$ is effective and does not contain $C$ in its support, we have
	\[
	-K_X\cdot C\leq -\psi^*K_{\PP^n}\cdot C-E\cdot C=-K_{\PP^n}\cdot\ell-E\cdot C\leq n=\dim X.
	\]
	
	The concluding statement follows from the fact that $\ell$ is smooth at $\psi(P)$ and therefore $\psi(P)$ is a $v$-adic limit of $k$-rational points.
	
\end{proof}

In the remainder of this subsection we give a brief overview of the existence of rational curves with small intersection against the anticanonical divisor.  We first discuss the situation over an algebraically closed field of characteristic $0$, and then pass to a number field.

\subsubsection{Finding rational curves: algebraically closed field}
There is an extensive literature on the lengths of extremal rays over an algebraically closed field of characteristic $0$.  We highlight two such results:
\begin{itemize}
\item Suppose that $X$ is a smooth projective variety.  In \cite{Mo}, Mori shows that every $K_{X}$-negative extremal ray is spanned by a rational curve $C$ satisfying $-K_{X} \cdot C \leq \dim(X) + 1$.
\item Suppose that $X$ is a $\mathbb{Q}$-factorial klt variety.  In \cite{Ka91}, Kawamata shows that $K_{X}$-negative extremal ray is spanned by a rational curve $C$ satisfying $-K_{X} \cdot C \leq 2\dim(X)$.
\end{itemize}

The following question asks whether there is a better bound on the lengths of extremal rays for varieties with terminal singularities.

\begin{question}
Let $X$ be a terminal $\mathbb{Q}$-factorial projective variety of dimension $n$ over an algebraically closed field of characteristic $0$.  Suppose that $\varphi: X \to X'$ is a contraction of a $K_{X}$-negative extremal ray $\mathcal{R}$.
\begin{enumerate}
\item For every point $P$ in the $\varphi$-exceptional locus, is there a rational curve $C$ through $P$ that generates $\mathcal{R}$ and satisfies $-K_{X} \cdot C \leq n + 1$?
\item Suppose that $X \not \cong \mathbb{P}^{n}$.  For every point $P$ in the $\varphi$-exceptional locus, is there a rational curve $C$ through $P$ that generates $\mathcal{R}$ and satisfies $-K_{X} \cdot C \leq n$?
\end{enumerate}
\end{question}

Question (1) has an affirmative answer in dimension $\leq 3$ due to \cite{CT}.  For threefolds Question (2) is known for ``many'' types of contractions due to \cite{Be}, \cite{Ka05}.  Both questions are answered affirmatively for split toric varieties in \cite{MS}.  

\subsubsection{Finding rational curves: number fields}
Finding rational curves over a number field is more difficult.  To the best of our knowledge the following question is open:

\begin{question}
Let $X$ be a smooth del Pezzo surface over a number field $k$.  Suppose that $P \in X(k)$ is a $k$-rational point.  Is there a curve $C$ birational to $\P^1_k$ through $P$ satisfying $-K_{X} \cdot C \leq 6$?  
\end{question}

\section{Surfaces}\label{sec:surfaces}

We now specialize to the case of split surfaces.  We begin with a definition.

\begin{definition}\label{dfn:split}
	A smooth algebraic surface $X$ is said to be {\em split} over a field $k$ if the Galois action on $\Pic(X)$ is trivial.
\end{definition}

As noted in the introduction, this definition is slightly more general than some others, e.g.~it includes Brauer-Severi varieties:~although the line bundle $\O(1)$ admits no nonzero rational sections, the Galois action on $\Pic(X)$ is still trivial.  

We now prove the main theorem of this paper.


%


\begin{proof}[{Proof of Theorem \ref{thm:generalsurface}}]
If $X$ has non-negative Kodaira dimension, then this result is essentially Theorem~4.2 of \cite{M2}, and indeed one need not assume that the surface is split.  The hypothesis in \cite{M2} is the truth of Conjecture~\ref{conj:vstrong}, but all that is used in the proof is essential boundedness at $P$.

Thus, we may assume that $X$ has negative Kodaira dimension.  If $X$ is geometrically ruled over a curve $B$ of positive genus, then $\alpha_{ess}(P,A)=\infty$ for all ample divisors $A$ simply because the rational points of $X$ are not Zariski dense.  We immediately conclude that Conjecture \ref{conj:ratcurve} is vacuously true for $X$.

This reduces to the case where $X$ is a rational surface.  (Since $X$ is split, it admits over $k$ a birational map to a Brauer-Severi surface, which since $X$ has a rational point must be isomorphic to $\P^2$.)  

We almost run the MMP for $K_X+A$ on $X$.  By ``almost'' we mean that we run the $(K_X+A)$-MMP until we obtain a $(K_X+A)$-negative birational contraction $\varphi\colon X\to Y$ from $X$ to a surface $Y$ such that the next step of the $(K_X+A)$-MMP contains $P'=\varphi(P)$ in the exceptional locus.

If $Y\cong\P^2_k$, then we obtain a birational morphism $\varphi'\colon X\to\P^2$ such that $\varphi'$ is an isomorphism near $P$.  In this case, Lemma~\ref{l:div-contraction-case-pn} immediately implies that the theorem is true.  

Thus, we may assume that $Y$ is not isomorphic to $\P^2_k$. If $P'$ lies on a $(-1)$-curve, let $C'$ be that $(-1)$-curve. If $P$ does not lie on a $(-1)$-curve, then $Y$ is isomorphic to a Hirzebruch surface (remember we dispatched the $\P_k^2$ case earlier!), so let $C'$ be the fibre through $P$ of a morphism $\pi\colon Y\to\P^1$.

In both cases, we have 
\begin{equation*}
	(K_{X'} + \varphi_{*}A) \cdot C' \leq 0 \qquad \qquad -K_{X'} \cdot C' \leq \dim(X).
\end{equation*}
because $C'$ is contracted by the next step of the $(K_X+A)$-MMP.  

But these two inequalities immediately imply that $\varphi_*A\cdot C'\leq\dim X$, which means that Theorem~\ref{theo:mmpcomputedonsubvar} finishes the proof.  
\end{proof}

\end{document}